\documentclass[reqno,10pt]{amsart}
\usepackage{color}
\usepackage{amsmath}
\usepackage{amsfonts}
\usepackage{amssymb}
\usepackage{amsthm}
\usepackage{newlfont}
\usepackage{graphicx}
\newtheorem{thm}{Theorem}[section]
\newtheorem{cor}[thm]{Corollary}
\newtheorem{lem}[thm]{Lemma}

\theoremstyle{definition}

\theoremstyle{remark}
\newtheorem{rem}[thm]{Remark}
\numberwithin{equation}{section}

\newcommand{\ed}{\end {document}}

\title[On Bernstein Inequality]
{ on a frequency localized Bernstein inequality and some generalized Poincare-type inequalities}

\author[D. Li]{Dong Li}
\address[D. Li]{Department of Mathematics, University of British Columbia, Vancouver BC Canada V6T 1Z2\\
and Institute for Advanced Study, 1st Einstein Drive, Princeton, NJ 08544, USA}%
\email{mpdongli@gmail.com}

\begin{document}
\begin{abstract}
We consider a frequency localized Bernstein inequality for the
fractional Laplacian operator which has wide applications in fluid
dynamics such as dissipative surface quasi-geostrophic equations. We
use a heat flow reformulation and prove the inequality for the full
range of parameters and in all dimensions. A crucial observation
is that after frequency projection
 the zero frequency part of the L\'{e}vy semigroup does not participate in
 the inequality and therefore can be freely
adjusted. Our proof is based on this idea and a
careful perturbation of the L\'{e}vy semigroup near the zero frequency
which preserves the positivity  and improves the time decay. Several
alternative proofs (with weaker results) are also included. As an application we
also give new proofs of some generalized Poincare type inequalities.

\end{abstract}

\maketitle

\section{Introduction}
In this note we consider a frequency localized Bernstein-type inequality which
has useful applications in fluid dynamics. Let $\alpha>0$ and consider the fractional Laplacian operator
$|\nabla|^\alpha$ defined via Fourier transform by the relation
\begin{align*}
\widehat{|\nabla|^\alpha f }(\xi)=(2\pi |\xi|)^\alpha \hat f(\xi), \quad \xi \in \mathbb R^d.
\end{align*}
Here $\hat f(\xi)$ is the usual Fourier transform of a scalar-valued function $f$ on $\mathbb R^d$.
To fix the notations, we adopt the following conventional definition of Fourier transform pair:
\begin{align}
\text{Fourier transform:}\qquad &(\mathcal F f)(\xi)=\hat f(\xi) = \int_{\mathbb R^d} f(x) e^{-2\pi i x \cdot \xi} dx, \notag \\
\text{Inverse Fourier transform:}\quad & f(x) =  \int_{\mathbb R^d}
\hat f(\xi) e^{2\pi i x \cdot \xi}d \xi. \notag
\end{align}
Occasionally we also use the notation $\mathcal F^{-1}$ to denote inverse Fourier transform.
Let $0 < \alpha \le 2$ and $1<q<\infty$, the Bernstein-type inequality we are interested in
takes the following form: for any $A_2>A_1>0$ and any $f \in L^q(\mathbb R^d)$ with
\begin{align} \label{sup_cond1}
\text{supp}(\hat f) \subset \{ \xi:\;
A_1 \le |\xi| \le A_2\},
\end{align}
 there is a constant $C$ depending only on $(d,p,\alpha,A_1,A_2)$ such that
\begin{align}
C \| f\|_{q}^q\ge \int_{\mathbb R^d} (|\nabla|^{\alpha} f) |f|^{q-2} f dx \ge \frac 1 {C} \| f\|_{q}^q. \label{Bern1}
\end{align}
Here $\|f \|_q$ is the usual Lebesgue norm of $f$ on $\mathbb R^d$. An equivalent and more commonly used formulation of
\eqref{Bern1} is stated in Corollary \eqref{cor1} in which the frequency support condition \eqref{sup_cond1}
is replaced by the Littlewood-Paley operators. In \eqref{Bern1}, the upper bound is trivial: it is a consequence of the
H\"older's inequality and the usual Bernstein inequality. As for the lower bound, the case $q=2$ is a simple consequence
of the Plancherel theorem since by \eqref{sup_cond1},
\begin{align*}
\int_{\mathbb R^d} |\nabla|^{\alpha} f f dx & = \int_{A_1 \le |\xi| \le A_2} (2\pi|\xi|)^{\alpha} |\hat f(\xi)|^2
d\xi \notag \\
& \ge  {(2\pi A_1)^{\alpha}} \int_{\mathbb R^d} |\hat f(\xi)|^2 dx = {(2\pi A_1)^{\alpha}} \| f \|_2^2.
\end{align*}
It is the case $1<q<\infty$, $q\ne 2$ which requires more elaborate analysis.

For the full Laplacian case $\alpha=2$, the inequality \eqref{Bern1} can be reduced
to the  form (still under the condition \eqref{sup_cond1})
\begin{align}
\int_{\mathbb R^d} | \nabla f |^2 |f|^{q-2} dx \ge C \| f \|_q^q \label{Bern2}
\end{align}
after an integration by parts argument. The inequality \eqref{Bern2} was first proved by Danchin
\cite{Dan1} when $q$ is an even integer and under a certain $q$-dependent small angle condition on the frequency
support. Planchon
\cite{Pl1} proved the case $\alpha=2$, $2<q<\infty$ by using an integration by parts argument.
In \cite{Dan2}, Danchin settled the remaining case $\alpha=2$, $1<q<2$ in the appendix of
that paper (see Lemma A.5 therein). The fractional Laplacian formulation of \eqref{Bern1} for $0<\alpha<2$ first
appeared in Wu \cite{Wu1} and it is of fundamental importance in the wellposedness theory for the dissipative quasi-geostrophic
equations. In \cite{CMZ07}, Chen, Miao and Zhang proved the inequality \eqref{Bern1} for $0<\alpha<2$,
$2<q<\infty$ by using an interpolation definition of Besov spaces. Recently Hmidi \cite{Hmidi11}
even generalized \eqref{Bern1} to some logarithm-damped fractional Laplacian operators of the form
\begin{align*}
\frac{|\nabla|^{\gamma}} {\log^{\beta}(\lambda +|\nabla|)}, \qquad 0\le \beta\le \frac 12,\; 0\le \gamma \le 1,
\lambda \ge e^{\frac{3+2\alpha}{\beta}}
\end{align*}
in dimensions $d=1,2,3$. The maximum principle for these nonlocal operators is obtained in \cite{Hmidi11} and \cite{DL12}.

The purpose of this note is to give a completely new proof of \eqref{Bern1} which works for $0<\alpha \le 2$ and for all
$1<q<\infty$. We begin by reformulating \eqref{Bern1} in terms of a (fractional) heat flow estimate.
The following result is the key step. See Remark \ref{rem_weak} for a slightly weaker result.

\begin{thm}[Improved heat flow estimate] \label{thm0}
Let the dimension $d\ge 1$. Let $0 < \alpha < 2$ and $1\le q  \le \infty$. There exists a constant $c>0$ depending {only
on the dimension $d$ and $\alpha$} such that
for any dyadic $N>0$ and any function $f\in L_x^q (\mathbb R^d)$, we have
\begin{align}
\|e^{-t|\nabla|^{\alpha}} P_N f \|_{q} \le e^{-c t N^{\alpha} } \| P_N f \|_q, \qquad \forall\, t\ge 0. \label{e00}
\end{align}
Here $P_N$ is the Littlewood-Paley operator defined in \eqref{lp_def}. For $\alpha=2$, there is an absolute constant $\tilde c>0$ such that
for any $1<q<\infty$, $f \in L_x^q(\mathbb R^d)$, we have
\begin{align}
\|e^{t\Delta} P_N f \|_{q} \le e^{-\tilde c \frac{q-1} {q^2} t N^2 } \| P_N f \|_q, \qquad \forall\, t\ge 0. \label{e00a}
\end{align}
\end{thm}

\begin{rem}
The usual Young's inequality together with the fact $\|\mathcal F^{-1} (e^{-t (2\pi|\xi|)^{\alpha}}) \|_{L^1_x}=1$
easily yields that
\begin{align}
\|e^{-t|\nabla|^{\alpha}} P_N f \|_{q} \le \| P_N f\|_q, \qquad\forall\, t \ge 0. \notag
\end{align}
The inequalities \eqref{e00}--\eqref{e00a} give a strengthening (thus the name "improved") of the above estimate.
It is of course fairly easy to prove the estimate
\begin{align}
\|e^{-t|\nabla|^{\alpha}} P_N f \|_{q} \le C_1 e^{-c t N^{\alpha} } \| P_N f \|_q, \qquad \forall\, t\ge 0
\end{align}
with a non-sharp constant $C_1$.  The main point of \eqref{e00} is that $C_1$ can take the sharp value $1$. This is
very important for deriving the later inequality \eqref{Bern1}.
We shall only need the estimate near $t=0$ to prove
the inequality \eqref{Bern1}. Also it is worthwhile pointing it out that in \eqref{e00} $P_N$ can
be replaced by $P_{\ge N}$ or $P_{>N}$ since the main property
needed in the proof is a certain spectral gap condition.
\end{rem}

\begin{rem}
We stress that the two bounds \eqref{e00} and \eqref{e00a} are essentially optimal.
In particular the constant $c$ in \eqref{e00} cannot be taken to be uniform for all $0<\alpha\le 2$ and will actually blow up at $\alpha=2$.
This is deeply connected with the fact that the decay of $e^{-t|\nabla|^{\alpha}}$ is power-like only for $0<\alpha<2$.
For the full Laplacian, even with frequency localization, one \emph{should not} expect the inequality
\begin{align*}
 \| e^{t \Delta} P_N f\|_{\infty} \le e^{-ctN^2} \| P_N f\|_{\infty}.
\end{align*}
To see this point it suffices to consider the periodic case, see Remark \ref{rem_counter_p} below.
\end{rem}

\begin{rem} \label{rem_weak}
 If we do not care so much about the constant dependence on $q$, we can give a much shorter (and almost trivial) proof.
 We sketch the argument as follows. By interpolating
 the obvious inequalities (here $0<\alpha\le 2$, $N>0$, and $c_1$ is an absolute constant):
 \begin{align}
  \|  e^{-t|\nabla|^{\alpha} } P_N f \|_2 &\le e^{-c_1t N^{\alpha}} \| P_N f \|_2,  \notag \\
 \|  e^{-t|\nabla|^{\alpha} } P_N f \|_{q} &\le \| P_N f \|_{q}, \quad q=1\, \text{or}\, \infty, \label{einfty_weak}
 \end{align}
we obtain
\begin{align} \label{erem_00a}
 \| e^{-t |\nabla|^{\alpha} } P_N f \|_q \le e^{-c_1 tN^{\alpha} \cdot \frac{q-1} {q^2} } \|P_N f\|_q, \quad \forall\, 1<q<\infty.
\end{align}
Comparing \eqref{erem_00a} with \eqref{e00}, the main improvement there is at the endpoints $q=1$ and $q=\infty$ for $0<\alpha<2$.
\end{rem}

\begin{rem}
 One may wonder whether it is possible to absorb the frequency localization into the kernel and prove directly the bound
 (say for $N=1$)
 \begin{align} \label{eNov29_1}
  \left \| \mathcal F^{-1} ( P_1 e^{-t |\nabla|^{\alpha} } ) \right\|_{L_x^1} \le e^{-ct}, \quad t>0.
 \end{align}
We show that \eqref{eNov29_1} is impossible even for $t$ sufficiently small. Let $\phi_1(\xi) = \varphi(\xi) -\varphi(2\xi)$ (see \eqref{lp_def} for
the definition of $\varphi$)
and consider the function
\begin{align*}
 g(t,x) =  \int_{\mathbb R^d} e^{-t(2\pi|\xi|)^{\alpha}} \phi_1(\xi) e^{2\pi i \xi \cdot x} d\xi.
\end{align*}
For $ |\xi| =1$, we have by definition
\begin{align*}
 1= |\phi_1(\xi)| & = \left| \int_{\mathbb R^d} g(0,x) e^{-2\pi ix \cdot \xi} d\xi \right| \notag \\
 & \le \int_{\mathbb R^d} |g(0,x)| |\cos (2\pi x \cdot \xi) | dx \notag \\
 & \le \int_{\mathbb R^d} |g(0,x) | dx.
\end{align*}
By examining the conditions for equality, it is not difficult to disprove the possibility $\| g(0,\cdot) \|_1=1$. Therefore
we have $\|g(0,\cdot)\|_1>1$. Since $\| g(t,\cdot)\|_1$ is continuous in $t$, we get for $\| g(t,\cdot)\|_1 >1$ for $t$ sufficiently small.
This disproves \eqref{eNov29_1}.
\end{rem}

The Bernstein inequality \eqref{Bern1} can be regarded as an infinitesimal version of the decay
estimate \eqref{e00}--\eqref{e00a}. We state it as the following corollary.

\begin{cor} \label{cor1}
Let the dimension $d\ge 1$. Let $0<\alpha<2$ and $1< q <\infty$. Then for any dyadic $N>0$,
 any $f\in L_x^q(\mathbb R^d)$,
 \begin{align}
 \int_{\mathbb R^d} (P_N |\nabla|^\alpha f )|P_N f |^{q-2} P_N f dx
 \ge c  N^{\alpha} \| P_N f\|_q^q, \label{e30a}
 \end{align}
 where the constant $c$ depends only on the dimension $d$ and $\alpha$.
For $\alpha=2$, there is an absolute constant $\tilde c>0$ such that for any $1<q<\infty$, any dyadic $N>0$ and
 any $f\in L_x^q(\mathbb R^d)$, we have the inequality
\begin{align}
- \int_{\mathbb R^d} (P_N \Delta f )|P_N f |^{q-2} P_N f dx
 \ge \tilde c \frac {q-1} {q^2} N^{2} \| P_N f\|_q^q. \label{e30b}
 \end{align}
\end{cor}

\begin{rem}
In \eqref{e30a} $P_N$ can be replaced by $P_{\ge N}$ or $P_{>N}$ or other similar frequency projection operators.
Note  that in Corollary \ref{cor1} the restriction
of $q$ is $1<q<\infty$. This is because we shall deduce
\eqref{e30a}--\eqref{e30b} from \eqref{e00}--\eqref{e00a} through a differentiation
argument. A rigorous justification of differentiating under the integral
requires $1<q<\infty$.
\end{rem}

\begin{rem}
 We stress that for $0<\alpha<2$ the constant $c$ in \eqref{e30a} depends only on $(d,\alpha)$. In particular it does not depend on
 the constant $q$. This is in sharp contrast with the full Laplacian case where the constant is proportional to $(q-1)/q^2$ which
 vanishes at $q=1$ and $q=\infty$.  In \cite{Dan2} (see equation (88) on page 1228 therein),
 Danchin effectively proved the inequality \eqref{e30b} by using an integration by parts argument. Our new proof here also reproduces
 the same constants. It should be possible to show that  \eqref{e30a}--\eqref{e30b} are essentially optimal. But we will not dwell on this
 issue here.
\end{rem}

\begin{rem}
 If we do not insist on obtaining the sharp constant dependence on $q$ (especially for $0<\alpha<2$), we can give a much shorter
 proof of a weaker version
 of  \eqref{e30a} which includes \eqref{e30b} as a special case. The starting point is the almost trivial inequality \eqref{erem_00a}.
By a rigorous differentiation and comparison argument at $t=0$ (see e.g. \eqref{pt_50c}--\eqref{e32} in the proof of Corollary \ref{cor1}), we
 arrive at the inequality
 \begin{align}
 \int_{\mathbb R^d} (P_N |\nabla|^{\alpha} f )|P_N f |^{q-2} P_N f dx
 \ge c^{\prime} \cdot \frac{q-1} {q^2} N^{\alpha} \| P_N f\|_q^q, \label{e30_new}
 \end{align}
 where $1<q<\infty$ and $c^{\prime} >0$ is an absolute constant. This is already enough for most applications in the local wellposedness theory
 of PDEs with fractional Laplacian dissipation.
\end{rem}

 Before we move on to other similar results, let us explain the ``mechanism'' of our proof.
 In some sense our proof is an upgraded version of the proof in Remark \ref{rem_weak}. In particular \eqref{e00} is a strengthened
 version of \eqref{einfty_weak} in the case $0<\alpha<2$ (and hence the improvement).  The proof of
\eqref{e00} is based
on frequency localization and Young's inequality. We briefly
explain the idea as follows. First notice that by scaling
it suffices to prove \eqref{e00} for $N=1$. By frequency
localization we have
\begin{align*}
e^{-t (2\pi|\xi|)^{\alpha}} \widehat{P_1 f}(\xi) &= e^{-t(2\pi|\xi|)^{\alpha}}
\psi(\xi) \hat f(\xi)  \notag \\
&=e^{-t (  (2\pi|\xi|)^{\alpha} + \epsilon \phi_1(\xi))}
\psi (\xi) \hat f(\xi),
\end{align*}
where $\psi(\xi)=\varphi(\xi)-\varphi(2\xi)$ (see \eqref{lp_def}) and $\phi_1(\xi)= \varphi(6 \xi)$.
In the last equality we used the fact that $\psi(\xi)=0$ for $|\xi| \le 1/2$
and $\phi_1(\xi)=0$ for $|\xi| \ge 1/3$. Now to prove \eqref{e00} it suffices to
show that the modified kernel
\begin{align*}
k_\epsilon(t,x) = \mathcal F^{-1} ( e^{-t ((2\pi|\xi|)^{\alpha} + \epsilon \phi_1(\xi))})
\end{align*}
has the bound $\| k_\epsilon(t,\cdot)\|_{L_x^1} \le e^{-ct}$ for some $c>0$. Since
$\hat k_\epsilon(t, 0)= e^{-t \epsilon \phi_1(0)} = e^{-t \epsilon}$,
we only need to prove that $k_{\epsilon}(t,x)$ is non-negative for $\epsilon$ sufficiently small.
The proof of this fact is given in Lemma \ref{lem0}. The main idea is to use the slow (power-like) decay (see \eqref{slow_decay})
of the L\'{e}vy semigroup when $0<\alpha<2$ which is stable under smooth perturbations.

It is fairly interesting to establish some analogues of Theorem \ref{thm0} and Corollary \ref{cor1} in the periodic
domain case. To fix notations, let the dimension $d\ge 1$ and $\mathbb T^d=\mathbb R^d/\mathbb Z^d$ be the usual periodic
torus. For a smooth periodic function $f:\, \mathbb T^d \to \mathbb R$, we adopt the Fourier expansion and
inversion formulae:
\begin{align*}
 f(x) & = \sum_{k\in \mathbb Z^d} \hat f(k) e^{2\pi i k\cdot x}, \quad x \in \mathbb T^d, \notag \\
 \hat f(n) &= \int_{\mathbb T^d} f(x) e^{-2\pi i n\cdot x} dx, \quad n \in \mathbb Z^d.
\end{align*}
For $0<\alpha \le 2$, the fractional Laplacian operator $|\nabla|^{\alpha}$ is defined by the relation
\begin{align*}
 \widehat{|\nabla|^{\alpha} f} (n) = (2\pi|n|)^{\alpha} \hat f(n),\quad n \in \mathbb Z^d.
\end{align*}
We shall say a function $f\in L^1(\mathbb T^d)$ has mean zero if $\hat f(0)=0$ or equivalently
\begin{align*}
 \int_{\mathbb T^d} f(x) dx =0.
\end{align*}

With these notations, we have

\begin{thm}[Improved heat flow estimate, periodic case] \label{thm0_period}
 Let the dimension $d\ge 1$. For any $0<\alpha < 2$, there is a constant $c_1>0$ depending only
 on $(\alpha,d)$ such that for any $1\le q\le \infty$ and any $f\in L^q(\mathbb T^d)$ with mean zero, we have
 \begin{align}
  \| e^{-t |\nabla|^{\alpha} } f \|_q \le e^{-c_1 t} \| f\|_q, \quad \forall\, t>0. \label{eNov30_1}
 \end{align}
For $\alpha=2$, there is an absolute constant $c_2>0$ such that for any $1<q<\infty$, $f\in L^q(\mathbb T^d)$ with
mean zero, we have
\begin{align}
 \| e^{t\Delta} f\|_q \le e^{- c_2 \frac{q-1} {q^2} t} \| f\|_q, \quad \forall\, t>0. \label{eNov30_2}
\end{align}
\end{thm}

\begin{rem}
 One should notice again the subtle difference between the bounds \eqref{eNov30_1} for $0<\alpha<2$ and \eqref{eNov30_2} for $\alpha=2$.
 The main reason is that in the periodic setting our perturbation argument also relies heavily on the pointwise lower bound of the
 L\'{e}vy semigroup for sufficiently small $t$. When $0<\alpha<2$ the decay of $e^{-t|\nabla|^{\alpha}}$ is power-like. However
 for $\alpha=2$ this is no longer the case and the perturbation argument does not work.
\end{rem}

\begin{rem} \label{rem_counter_p}
 We stress that \eqref{eNov30_2} is optimal. In particular one should not expect the inequality
 \begin{align}
  \| e^{t\Delta} f \|_\infty \le e^{-ct} \| f\|_{\infty},  \label{erem_impossible}
 \end{align}
even for $t>0$ sufficiently small. To see this, we take any smooth $f$ on $\mathbb T^d$ with $\| f\|_{\infty} =1$ and zero mean.
Suppose $x_0\in \mathbb T^d$ and  $f(x)\equiv 1$ in some neighborhood $|x-x_0|\le \delta_0$ with $\delta_0>0$.  Write
\begin{align*}
 e^{t\Delta} f = k(t,\cdot)*f,
\end{align*}
where $*$ denote the usual convolution on $\mathbb T^d$ and $k(t,\cdot)$ is the periodic heat kernel. By using the Poisson summation formula
(see Lemma
\ref{lem_poi}),
it not difficult to check that for $|y|\le \frac 12$, $0<t<1$,  we have
\begin{align*}
 k(t,y)= \frac 1{(4\pi t)^{\frac d 2} } e^{-\frac{|y|^2}{4t}} + O(e^{-\frac C t}),
\end{align*}
where $C>0$ is some constant.  We then have for $t>0$ sufficiently small and some constant $C_{0}>0$,
\begin{align*}
 |(e^{t\Delta} f)(x_0)| &\ge \int_{|y|\le  \frac {\delta_0} 2} \frac 1 {(4\pi t)^{\frac d2} } e^{-\frac{|y|^2}{4t} } dy +
 O(e^{-\frac {C} t}) \\
 & \ge 1 + O( e^{-\frac {C_0} t} ).
\end{align*}
This disproves \eqref{erem_impossible}.

\end{rem}

An immediate consequence of Theorem \ref{thm0_period} is a family of generalized Poincare-type inequalities.

\begin{cor}[Generalized Poincare-type inequalities for periodic domains] \label{cor_period}
 Let the dimension $d\ge 1$.   For any $0<\alpha < 2$, there is a constant $c_1>0$ depending only
 on $(\alpha,d)$ such that for any any $f\in C^{2}(\mathbb T^d)$ with mean zero, we have
 \begin{align}
   \int_{\mathbb T^d} (|\nabla|^{\alpha} f) |f|^{q-2} f dx \ge c_1 \| f\|_q^q, \quad \forall\, 1<q<\infty. \label{eNov30_3}
 \end{align}
For $\alpha=2$, there is an absolute constant $c_2>0$ such that for any  $f\in C^{2}(\mathbb T^d)$ with
mean zero, we have
\begin{align}
 - \int_{\mathbb T^d} (\Delta f) |f|^{q-2} f dx \ge c_2 \frac{q-1} {q^2}\| f\|_q^q, \quad \forall\, 1<q<\infty. \label{eNov30_4}
\end{align}
\end{cor}
The proof of Corollary \ref{cor_period} is quite similar to the proof of Corollary \ref{cor1} and therefore we omit it.

\begin{rem}
 In \cite{KS_book} (see Proposition A.14.1 on page 291 therein), by using a contradiction argument,
 the authors proved the inequality \eqref{eNov30_4} for the case $2\le q<\infty$
 with a dimension-dependent constant. Our new proof here covers the whole range $1<q<\infty$ with a dimension-independent constant $c_2$.
\end{rem}

We conclude the introduction by setting up some

\subsubsection*{Notations}
We will  need to use the
Littlewood-Paley frequency projection operators. Let $\varphi(\xi)$ be a smooth bump
function supported in the ball $|\xi| \leq 2$ and equal to one on
the ball $|\xi| \leq 1$.  For each dyadic number $N \in 2^{\mathbb Z}$ we
define the Littlewood-Paley operators
\begin{align}
\widehat{P_{\leq N}f}(\xi) &:=  \varphi(\xi/N)\hat f (\xi), \notag\\
\widehat{P_{> N}f}(\xi) &:=  [1-\varphi(\xi/N)]\hat f (\xi), \notag\\
\widehat{P_N f}(\xi) &:=  [\varphi(\xi/N) - \varphi (2 \xi /N)] \hat
f (\xi). \label{lp_def}
\end{align}
Similarly we can define $P_{<N}$, $P_{\geq N}$, and $P_{M < \cdot
\leq N} := P_{\leq N} - P_{\leq M}$, whenever $M$ and $N$ are dyadic
numbers.

\subsection*{Acknowledgements}
The author would like to thank Prof. Ya.G. Sinai for informing him the book \cite{KS_book}.
D. Li was supported in part by NSF under agreement No. DMS-1128155. Any opinions, findings
and conclusions or recommendations expressed in this material are those of the authors and
do not necessarily reflect the views of the National Science Foundation. D. Li was also supported in part
by an Nserc discovery grant.
\section{Proof of main theorems}

We begin with a simple lemma.  Let $\phi_1 \in C_c^{\infty} (\mathbb R^d)$ be a radial function such that
\begin{align} \label{e_phi1}
\phi_1(x)=
\begin{cases}
 1, \quad |x|\le \frac 14, \\
0, \quad |x| \ge \frac 13.
\end{cases}
\end{align}
  Let $\epsilon>0$ and define for $t>0$,
\begin{align} \label{eq_F}
F_{\epsilon} (t,x) =
\int_{\mathbb R^d} e^{-t (2\pi|\xi|)^{\alpha}} { (e^{-\epsilon t \phi_1(\xi) } -1)}  e^{2\pi i \xi \cdot x} d\xi.
\end{align}
Also denote
\begin{align}
p(t,x) = \mathcal F^{-1} ( e^{-t(2\pi|\xi|)^{\alpha}}) =  \int_{\mathbb R^d}
e^{-t(2\pi |\xi|)^{\alpha}} e^{ 2\pi i\xi \cdot x} d\xi, \label{eq_P}
\end{align}
and
\begin{align}
k_{\epsilon}(t,x) = \mathcal F^{-1} ( e^{-t ((2\pi |\xi|)^{\alpha} + \epsilon \phi_1(\xi)} )
= \int_{\mathbb R^d}
e^{-t  ((2\pi |\xi|)^{\alpha} +\epsilon \phi_1(\xi) )} e^{2\pi i\xi \cdot x} d\xi. \label{eq_k}
\end{align}
Clearly
\begin{align} \label{e_23a}
k_{\epsilon}(t,x) = p(t,x) + F_{\epsilon}(t,x).
\end{align}

The following lemma shows that for $\epsilon>0$ sufficiently small $k_{\epsilon}(t,x)$ is still
a positive kernel.

\begin{lem} \label{lem0}

Let  $d\ge 1$ and $0<\alpha<2$. There exists a constant $C_1=C_1(d,\alpha)>0$ such
that  for any $x \in \mathbb R^d$, $t>0$,
\begin{align} \label{to1a}
\frac 1 {C_1}   \cdot \frac {t} { (t^{\frac 1 {\alpha}} + |x| )^{d+\alpha} }
\le p(t,x) \le {C_1}  \cdot \frac {t} { (t^{\frac 1 {\alpha}} + |x| )^{d+\alpha} } .
\end{align}
There exists a constant $\epsilon_0=\epsilon_0(d,\alpha)>0$
 such that if $0<\epsilon<\epsilon_0$, then
\begin{align} \label{e4_7}
p(t,x) - |F_{\epsilon}(t,x)| > 0, \qquad \forall\, x \in \mathbb R^d,  0<t \le 1.
\end{align}
In particular
\begin{align} \label{39a}
\| k_{\epsilon} (t,\cdot) \|_{L_x^1} \le e^{- \epsilon t}, \qquad \forall\, 0<t \le 1.
\end{align}

\end{lem}

\begin{proof}[Proof of Lemma \ref{lem0}]
The bound \eqref{to1a} is a well-known result, cf. \cite{BG60}:
\begin{align} \label{slow_decay}
\frac 1 {C_1} \min\{ t^{-d/\alpha},\; \frac t
{|x|^{d+\alpha}} \}
\le p(t,x) \le {C_1} \min\{ t^{-d/\alpha},\; \frac t
{|x|^{d+\alpha}} \}
\end{align}
which is clearly equivalent to \eqref{to1a}.  By \eqref{eq_F} we have
\begin{align}
F_{\epsilon}(t,x) = p(t)*w(t)= \int_{\mathbb R^d} p(t,x-y) w(t,y)dy, \label{e_conv}
\end{align}
where
\begin{align*}
w(t,y)&= \epsilon t \int_{\mathbb R^d}  \frac {e^{-\epsilon t \phi_1(\xi)} -1}
{\epsilon t} e^{2\pi i \xi \cdot y} d\xi \notag \\
&=: \epsilon t \, w_1(t,y).
\end{align*}
It is not difficult to check that for any $0<\epsilon t \le 1$, we have
\begin{align*}
|w_1(t,y)| \le C(d,\alpha)\cdot \frac 1 {(1+|y|^2)^{d+10}}, \qquad \forall\, y \in \mathbb R^d.
\end{align*}
Comparing this with \eqref{to1a},  we get
\begin{align*}
|w(t,y)| \le C(d, \alpha) \cdot \epsilon\cdot  p(t,y), \qquad \forall\, y \in \mathbb R^d, \, 0<t \le 1.
\end{align*}
Plugging it into \eqref{e_conv} and using the fact that $p(t)*p(t)=p(2t)$,  we get
\begin{align}
F_{\epsilon}(t,x) &\le C(d,\alpha) \cdot \epsilon p(2t,x), \notag \\
& \le \frac 12 p(t,x), \qquad \forall\, x \in \mathbb R^d, \, 0<t \le 1, \notag
\end{align}
where we used the fact $p(2t,x) \le const \cdot p(t,x)$ and chose $\epsilon$ sufficiently small.
Clearly \eqref{e4_7} follows. Finally \eqref{e4_7} and
\eqref{e_23a} implies that $k_{\epsilon}(t,x)$ is positive everywhere.  The bound
\eqref{39a} follows from Fourier transform and the fact that $\phi_1(0)=1$.
\end{proof}

Now we are ready to complete the
\begin{proof}[Proof of Theorem \ref{thm0}]
First we note that \eqref{e00a} is already proved in Remark \ref{rem_weak}. Therefore we only need to prove
\eqref{e00} for $0<\alpha<2$.

By a scaling argument
we only need to prove \eqref{e00} for $N=1$. In view of the usual convolution property of the heat semigroup, i.e.
\begin{align*}
 e^{-t|\nabla|^{\alpha} } = \Bigl( e^{- \frac t {m} |\nabla|^{\alpha}} \Bigr)^m,
\end{align*}
it suffices to prove
\eqref{e00} for $0<t\le 1$.
Now recall that
\begin{align*}
\widehat{P_1 f}(\xi) = \psi(\xi) \hat f(\xi),
\end{align*}
where $\psi$ is compactly supported in the annulus $\{\xi:\; 1/2 \le |\xi| \le 2 \}$.  In view of this
localization property, we can smoothly redefine the kernel $e^{-t(2\pi|\xi|)^{\alpha}}$ on the tail part $\{\xi:\; |\xi| < \frac 12\}$
such that
\begin{align} \label{id_1}
e^{-t(2\pi|\xi|)^{\alpha}} \widehat{P_1 f}(\xi) = e^{-t \left( (2\pi|\xi|)^{\alpha}+ \epsilon \phi_{1} (\xi) \right)} \widehat{P_1 f}(\xi),\qquad \forall\, \xi \in \mathbb R^d,
\end{align}
where we choose $\phi_{1} (\xi) $ as in \eqref{e_phi1}.
Now denote $\phi_{\epsilon}(\xi)=(2\pi |\xi|)^{\alpha} + \epsilon \phi_1(\xi)$.

Therefore we only need to prove for $\epsilon>0$ sufficiently small,
\begin{align*}
\left\| \Bigl(\mathcal F^{-1} (e^{-t \phi_{\epsilon} } ) \Bigr)
* P_1 f \right\|_q \le e^{-\epsilon t} \| P_1 f\|_q, \quad \forall\, 0< t \le 1.
\end{align*}

But this follows easily from Lemma \ref{lem0} and Young's inequality.
\end{proof}

To prove Corollary \ref{cor1}, we need the following simple lemma.

\begin{lem} \label{lem5}
Let $0<\alpha\le 2$. Let $f \in L^1_{loc}(\mathbb R^d)$. Then there is
a constant $C_{\alpha,d}>0$ depending only on $(\alpha,d)$, such that
\begin{align}
\sup_{0< t <\infty} | (e^{-t |\nabla|^{\alpha}} f) (t,x) | \le C_{\alpha,d} (Mf)(x),
\label{Mf_bound}
\end{align}
where $Mf$ is the Hardy-Littlewood maximal function defined as
\begin{align*}
(Mf)(x)= \sup_{B\ni x} \frac 1{|B|} \int_B |f|.
\end{align*}
\end{lem}

\begin{proof}[Proof of Lemma \ref{lem5}]
Denote $p(y)=\mathcal F^{-1}(e^{-(2\pi |\xi|)^{\alpha}})$. By \eqref{to1a} we have
for some constant $C_1>0$,
$$p(y) \le C_1 \cdot (1+|y|)^{-(d+\alpha)}.$$
Therefore for $0<t <\infty$,
\begin{align*}
| &(e^{-t|\nabla|^{\alpha}} f )(t,x) | \notag \\
& \le  \int_{\mathbb R^d} p(y)|f(x+t^{\frac 1 {\alpha}} y) | dy \notag\\
& \le \int_{|y| \le 1} p(y) |f(x+ t^{\frac 1 {\alpha}} y)| dy+ \sum_{k=1}^{\infty} \int_{2^{k-1} \le |y|
\le 2^k} p(y) |f(x+ t^{\frac 1 {\alpha}} y)| dy \notag \\
& \le C_1 \int_{|y| \le 1} |f(x+t^{\frac 1 {\alpha}} y)|dy+C_1 \sum_{k=1}^{\infty} 2^{-(k-1)(d+\alpha)}
\int_{|y| \le 2^k} |f(x+ t^{\frac 1 {\alpha}} y) | dy \notag \\
& \le C_1 2^{d+\alpha} \sum_{k\ge 0} 2^{-k\alpha} \sup_{k^{\prime} \ge 0}  \int_{|y| \le 1}
|f(x+ 2^{k^{\prime}} t^{\frac 1 {\alpha}} y )|dy \notag \\
& \le C_{d,\alpha} (Mf)(x).
\end{align*}
\end{proof}

\begin{proof}[Proof of Corollary \ref{cor1}]
Define
\begin{align*}
F(t) = \int_{\mathbb R^d} | e^{-t|\nabla|^{\alpha}} P_N f |^q dx
\end{align*}
Note that for each $t\ge 0$,
\begin{align*}
\partial_t ( e^{-t |\nabla|^{\alpha} } P_N f) & =  -e^{-t |\nabla|^{\alpha}} |\nabla|^{\alpha} P_N f.  \notag
\end{align*}
Since $1<q<\infty$, we get
\begin{align}
&\partial_t ( |e^{-t |\nabla|^{\alpha} } P_N f |^q)  \notag \\
=&\;  - q(e^{-t |\nabla|^{\alpha}} |\nabla|^{\alpha} P_N f)
\cdot |e^{-t|\nabla|^{\alpha}}P_N f|^{q-2} \cdot
e^{-t |\nabla|^{\alpha}} P_N f. \label{pt_50a}
\end{align}
By Lemma \ref{lem5} and \eqref{pt_50a}, we have
\begin{align}
\sup_{t\ge 0} &\left| \biggl(\partial_t (|e^{-t |\nabla|^{\alpha} } P_N f |^q) \biggr)(t,x) \right| \notag \\
& \le q M(|\nabla|^{\alpha}P_N f)(x) \cdot \left( M(f)(x) \right)^{q-1} \notag \\
& \le \left(M(|\nabla|^{\alpha}P_N f)(x)\right)^q + (q-1) \left( M(f)(x) \right)^{q} \notag \\
& =: H(x). \label{pt_50b}
\end{align}
Since $1<q<\infty$ and $f \in L^q_x(\mathbb R^d)$, it is easy to check that $H \in L_x^1(\mathbb R^d)$.

Now denote $b(t,x) = \partial_t ( |e^{-t |\nabla|^{\alpha}} P_N f|^q)(t,x)$. Observe that
by \eqref{pt_50a}, it is easy to check that $b$ is continuous in $(t,x)$ and consequently
\begin{align}
\lim_{\delta \downarrow 0} \frac {\int_0^\delta b(s,x) ds} \delta = -q (|\nabla|^{\alpha}P_N f)(x)
|P_N f(x)|^{q-2} P_N f(x), \qquad \forall\, x \in \mathbb R^d. \label{pt_50c}
\end{align}
Now let $0<h<1$ and write
\begin{align*}
 \frac{F(h)-F(0)} h
=  \int_{\mathbb R^d} \frac{ \int_0^h b(s,x) ds} h dx.
\end{align*}

By \eqref{pt_50b}, we have
\begin{align}
\sup_{0 <h <1} \Bigl|\frac{ \int_0^h b(s,x) ds} h\Bigr| \le H(x). \label{pt_50d}
\end{align}

By \eqref{pt_50c}--\eqref{pt_50d} and the Lebesgue Dominated Convergence Theorem, we obtain that $F$ is right-differentiable
at $t=0$ and
\begin{align*}
F^{\prime}(0+) = -q \int_{\mathbb R^d} (|\nabla|^{\alpha}P_N f)(x)
|P_N f(x)|^{q-2} P_N f(x) dx.
\end{align*}

In particular we have
\begin{align}
\text{LHS of \eqref{e30a}} = - \frac 1 q F^{\prime}(0+).  \label{e31}
\end{align}

On the other hand by using Theorem \ref{thm0}, we have for any $t>0$,
\begin{align*}
\frac {F(0) -F(t)} {t} \ge \frac {1-e^{-ctqN^{\alpha}}} t  \| P_N f \|_q^q.
\end{align*}
Taking the limit $t\to 0$ immediately give us
\begin{align}
- F^{\prime}(0) \ge c q N^{\alpha} \|P_N f \|_q^q. \label{e32}
\end{align}

By \eqref{e31} this gives us \eqref{e30a}.

\end{proof}

For the periodic case we recall the following standard Poisson summation formula.

\begin{lem}[Poisson summation formula] \label{lem_poi}
 Let $f$ be a Schwartz function on $\mathbb R^d$ and denote by $\hat f$ its Fourier transform on $\mathbb R^d$. Then
 \begin{align*}
  \sum_{n \in \mathbb Z^d} f(x+n) = \sum_{n \in \mathbb Z^d} \hat f(n) e^{2\pi i n \cdot x}, \quad \forall\, x \in \mathbb R^d.
 \end{align*}

\end{lem}

As is well known, the Poisson summation formula holds for general continuous functions of moderate decrease including the heat
kernels
\begin{align}
 k_{\alpha}(t,\cdot) = \mathcal F^{-1} ( e^{-t (2\pi |\xi|)^{\alpha}}), \quad 0<\alpha\le 2, t>0. \notag
\end{align}
Denote the periodic kernel
\begin{align}
 k_{\alpha}^{per} (t,x) = \sum_{n\in \mathbb Z^d} e^{-t(2\pi |n|)^{\alpha}} e^{2\pi i n\cdot x}. \notag
\end{align}

The next lemma gives a lower bound on $k_{\alpha}^{per}(t,x)$. It is amusing that later  this lower bound is used to prove
an upper bound of the heat kernel.

\begin{lem}[Short time lower bound of the fractional heat kernel] \label{lem_lower}
Let the dimension $d\ge 1$ and let $0<\alpha<2$. There is a constant $c_3>0$ depending only on $(\alpha,d)$ such that
\begin{align}
 k_{\alpha}^{per}(t,x) \ge c_3 t, \quad \forall\, 0<t\le 1, \, x \in \mathbb T^d. \label{eDe_1}
\end{align}

\end{lem}

\begin{proof}[Proof of Lemma \ref{lem_lower}]
 By Lemma \ref{lem_poi} and \eqref{to1a}, we have for $|x|\le 1$, $0<t\le 1$,
 \begin{align*}
  k_{\alpha}^{per}(t,x) & = \sum_{n \in \mathbb Z^d} k_{\alpha}(t,x+n) \notag \\
  & \ge \sum_{n \in \mathbb Z^d, \, |n|=3} k_{\alpha}(t,x+n) \notag \\
  & \ge c_3 t,
 \end{align*}
where $c_3$ depends only on $(d,\alpha)$.
\end{proof}

We now complete the
\begin{proof} [Proof of Theorem \ref{thm0_period}]
 We only need to show \eqref{eNov30_1} since \eqref{eNov30_2} follows the same argument as in \eqref{erem_00a}.
 Now assume $0<\alpha<2$. By the convolution property we only need to prove the case $0<t\le 1$.
 Since $\hat f(0)=0$ we may freely adjust the $\text{zero}^{th}$ Fourier coefficient of the heat kernel.
Doing so gives us
 \begin{align*}
  e^{-t|\nabla|^{\alpha}} f = \tilde k(t,\cdot)*f,
 \end{align*}
where
\begin{align*}
 \tilde k(t,x) = k_{\alpha}^{per}(t,x) -c_3 t,
\end{align*}
By Lemma \ref{lem_lower}, we have
\begin{align*}
 \tilde k(t,x) \ge 0, \quad \forall\, 0<t\le 1, \, x \in \mathbb T^d.
 \end{align*}
 Hence
 \begin{align*}
 \| \tilde k(t,\cdot) \|_1 = 1-c_3 t \le e^{-c_1 t}, \quad \forall\, 0<t \le 1.
\end{align*}
Obviously \eqref{eNov30_1} follows from the above bound and Young's inequality.
\end{proof}

\end{document}